\documentclass{article}
\usepackage{mathtools}
\usepackage{spconf,amsmath,graphicx, amsfonts, amssymb, amsthm}
\usepackage{algorithm}
\usepackage[noend]{algpseudocode}
\usepackage{epsfig,amsmath,color}

\newcommand{\R}{\mathbb R}
\newcommand{\Sc}{\mathcal S} 

\newcommand{\vertiii}[1]{{\left\vert\kern-0.25ex\left\vert\kern-0.25ex\left\vert #1 
    \right\vert\kern-0.25ex\right\vert\kern-0.25ex\right\vert}}
    
\DeclarePairedDelimiter\floor{\lfloor}{\rfloor}

\usepackage{enumitem} 

\DeclareMathOperator*{\argmin}{argmin}

\newtheorem{assumption}{Assumption} %
\newtheorem{theorem}{Theorem} 
\newtheorem{corollary}{Corollary} 
\newtheorem{remark}{Remark}

\title{Online Sparse Subspace Clustering}
\name{Liam Madden$^\dagger$, Stephen Becker$^\dagger$, Emiliano Dall'Anese$^\star$ 
}
\address{$^\dagger$Department of Applied Mathematics, University of Colorado Boulder \\
$^\star$Department of Electrical, Computer and Energy Engineering, University of Colorado Boulder}
\begin{document}
\maketitle
\begin{abstract}
This paper focuses on the  sparse subspace clustering problem, and develops an online algorithmic solution to cluster data points on-the-fly, without revisiting the whole dataset. The  strategy involves an online solution of a sparse representation (SR) problem to build a (sparse) dictionary of similarities where points in the same subspace are considered ``similar,'' followed by a spectral clustering based on the obtained similarity matrix. When the SR cost is strongly convex,  the online solution converges to within a neighborhood of the optimal time-varying batch solution.
A dynamic regret analysis is performed when the SR cost is not strongly convex. 
\end{abstract}
\begin{keywords}
Subspace clustering; sparse representation; time-varying optimization; online algorithms. 
\end{keywords}
\section{Introduction and Problem Setup}
\label{sec:intro}

Modern data processing tasks aim to extract information from datasets or signals on graphs --  examples include identification of trends or patterns, learning of dynamics and data structures, or methods for  comprehensive awareness of the underlying networks or systems generating the data~\cite{Becker18SPMag,Slavakis14SPMag}. In this domain, the present paper focuses on data processing methods for streams of (possibly high-dimensional) data, with particular emphasis on a setting where underlying computational constraints require one to process data on-the-fly, and with limited access to stored data. 

 One prominent task is clustering, which is utilized to cluster data points based on well-defined metrics modeling similarities (or distances) among data points, or capturing underlying data structures. For example, spectral clustering groups data points based on minimizing cuts of a similarity graph \cite{Lux}. In particular, subspace clustering builds a similarity graph where points in the same subspace are considered ``similar''. It does this by finding either a low-rank representation (LRR) \cite{LRR,Shen} or sparse representation (SSC) \cite{Vidal,Becker} of the data such that the data points are represented as linear combinations of other data points in the same subspace. In this paper, we will consider an \emph{online} SSC methodology, where underlying computational considerations prevent one from solving pertinent optimization problems associated with a given set of data~\cite{Vidal} before a new datum arrives.

To outline the problem concretely, consider $N$-dimensional data points $\{x_{t}, t \in \mathbb{N}\}$ sequentially arriving at times $ t\cdot h$, with $h > 0$ a given interval\footnote{\emph{Notation}: hereafter,  $(\cdot)^T$ denotes transposition. For a given  vector $x \in \mathbb{R}^N$ or matrix $X \in \mathbb{R}^{N\times M}$, 
$\|x\|$ and $\|X\|$ refer to a generic norm,
and $\vertiii{X}$ the spectral norm.
If $X_{ij}$ is the $(i,j)$ entry of $X$, then $\|X\|_F^2=\sum_{i,j} X_{ij}^2$ and $\|X\|_1 = \sum_{i,j} |X_{ij}|$.
The composition of two operators is denoted by $\circ$. Write $f(n)=\mathcal{O}(g(n))$ to denote that for sufficiently large $n$, $\exists c>0$ such that $|f(n)| \le c |g(n)|$.
}.
Assume that the data points lie in (or in the neighbourhood of) $S$ \emph{subspaces} $\{\Sc_i\}_{i=1}^S \subset \R^N$, with $\text{dim}(\Sc_i)=d_i$ for each $i$. 
This paper studies the problem of repeatedly applying SSC to all observed data up to a given time, $\{x_{j}, j \le t\}$ with $\bar{X}_t = [x_{1}, \ldots, x_t]$. SSC is a two-step approach \cite{Vidal}: first, step [S1], based on the self-expressiveness property, a sparse representation (SR) problem is solved to identify the (sparse) coefficients $\{\bar{c}_j\}_{j = 1}^t$ so that $x_j = \bar{X}_t \bar{c}_j$ for all $j = 1, \ldots, t$; that is, data points are represented as linear combinations of data points in the same subspace (and we force the $j^\text{th}$ component of $\bar{c}_j$ to be $0$ to exclude the trivial solution). Second, step [S2], apply spectral clustering based on a similarity matrix $W := |\bar{C}_t| + |\bar{C}_t^T|$, where $\bar{C}_t := [\bar{c}_1, \ldots \bar{c}_t]$ and  $|\cdot|$ is taken entry-wise. However, in the streaming setting, this setting has two drawbacks: 
\begin{enumerate}[wide, labelwidth=!, labelindent=6pt]
    \setlength\itemsep{.3em}
        \setlength\parskip{0pt}
        \setlength\parsep{0pt}
    \item[(d2)]  The dimensions  of $\bar{X}_t$ and $\bar{C}_t$ grow with time, thus increasing the complexity of the associated SR and spectral clustering tasks; and,
    \item[(d2)]  Due to underlying computational complexity considerations, 
steps [S1] and [S2] might not be executed to completion within a time interval $h$ (i.e., before a new datum arrives).
\end{enumerate}

Given (d1)-(d2), we address the problem of developing online algorithmic solutions to carry out steps [S1]-[S2] at each time $t$, based on a \emph{given computational budget}. The first step towards this goal involves the processing of  data points using a ``sliding window'' $X_t:=[x_{t-T+1}, \ldots, x_t]\in \R ^{N\times T}$ of length $T$ (with $T$ determined by the computational budget, as explained later in the paper). Step [S1] is ideally carried out at time $t$ by solving the following SR problem:
\begin{align} \label{eq:srt}
    C_t^* \in \argmin_{C\in\R^{T \times T}}&
    \hspace{.3cm} F_t(C)\equiv\|C\|_1+\frac{\lambda_e}{2}\|X_t-X_tC\|_{\text{F}}^2 \tag{SR$_t$}\\
    \text{s.t.}& \hspace{.3cm} \text{diag}(C)=0 \notag
\end{align}
with $\lambda_e > 0$ a given tuning parameter. 
If $\lambda_e$ is too small, in particular if $\lambda_e\le \|\textrm{vec}(X_t^T X_t)\|_{\infty}^{-1}$, then the minimizer $C_t^*$ may have all-zero columns, which is not informative, hence we always choose $\lambda_e$ sufficiently large.

Solving (SR$_t$) \emph{to convergence} within a time interval $h$ might not be possible for a given computational budget, especially for streams of high-dimensional vectors over a large window $T$. Section~\ref{sec:sparse} will address the design of \emph{online} algorithms, for the case where only one algorithmic step can be performed before a new datum $x_t$ arrives (the case of multiple steps follows easily). 

With  the minimizer (or approximate solution) $C_t^*$, one can compute the matrix $W_t=|C_t^*|+|C_t^*|^T$. Interpreting $W$ as the ``similarity'' matrix of a graph, one can compute the graph Laplacian $L_t = D_t - W_t$ where the degree matrix $D_t$ is a diagonal matrix attained by summing the rows of $W_t$.
The graph Laplacian is then normalized in one of two possible ways: as the symmetric graph Laplacian, $L_{sym,t}=D_t^{-1/2}L_tD_t^{-1/2}$, or as the random walk graph Laplacian, $L_{rw,t} = D_t^{-1} L_t$. We then compute the $S$ trailing eigenvectors of the normalized graph Laplacian and, viewing these eigenvectors as columns, we cluster their rows in $\R ^S$ into $S$ cluster using the k-means algorithm; see, e.g., \cite{Lux} for details on spectral clustering. This clustering is then applied to the original data points. 
Section~\ref{sec:spectral} will elaborate further on this step. 

\section{Online Sparse Representation}
\label{sec:sparse}

The proximal gradient descent algorithm and its accelerated version~\cite{Beck}
have rigorous convergence guarantees and can be applied to equations of the form \eqref{eq:srt}, which we now detail.
Let $f_t (C) = \frac{\lambda_e}{2}\|X_t-X_tC\|^2_F$ for brevity, and notice that $\nabla f_t(C)=\lambda_eX_t^T(X_tC-X_t)$ so $\nabla f_t$ is Lipschitz continuous with constant $M_t = \lambda_e\vertiii{X_t^TX_t}$. Let $n$ be the iteration index of the algorithm, and let $\gamma<\frac{2}{M_t}$. Then the (batch) proximal gradient descent algorithm, used to solve~\eqref{eq:srt} at time $t$, involves the following iterations for $n = 1, 2, \dots$ until convergence:  
\begin{equation}
\label{eq:prox_batch}
    C_{n+1}=\text{prox}_{\gamma\|\cdot\|_1,\text{diag}(\cdot)=0}\circ(I-\gamma\nabla f_t)(C_n)
\end{equation}
where $\text{prox}_{g,\mathcal{X}}(z)=\text{argmin}_{x\in \mathcal{X}}g(x)+\frac{1}{2}\|x-z\|^2$ is the proximal operator defined over a closed convex set $\mathcal{X}$ and for a function $g$. Convergence of~\eqref{eq:prox_batch} to a minimizer $C_t^*$ is shown in, e.g.,~\cite{Combettes}. Furthermore, $F_t(C_n)-F_t(C_t^*)\rightarrow 0$ by the continuity
of $F_t$; see Theorem 10.21 in \cite{Beck} for the rate. 

Consider now the case where only one iteration~\eqref{eq:prox_batch} can be performed per time interval $h$ (see Remark 3 for the extension to multiple steps).
Then, an \emph{online} implementation of the proximal gradient descent algorithm involves the sequential execution of the following step at each time $t$:
\begin{equation}
\label{eq:prox_online}
    C_{t+1}=\underbrace{\text{prox}_{\gamma_t\|\cdot\|_1,\text{diag}(\cdot)=0}}_{P_t}\circ\underbrace{(I-\gamma_t\nabla f_t)}_{G_t}(C_t)
    \vspace{-.9em}
\end{equation}
where the coefficient $\gamma_t$ is selected so that $\gamma_t<\frac{2}{M_t}$.
The difference between \eqref{eq:prox_batch} and \eqref{eq:prox_online} is that $f_t$ changes per iteration in the latter. 
The goal is to demonstrate that the online algorithm~\eqref{eq:prox_online} can \emph{track} the sequence of optimizers $\{C_t^*\}$. In the following, the performance of the online algorithm~\eqref{eq:prox_online} is investigated in two cases: 

\noindent i) The cost of~\eqref{eq:srt} is \emph{strongly convex} for each $t$. In this case, we will derive bounds for $\|C_t-C_t^*\|$, where $C_t^*$ is the \textit{unique} minimizer of~\eqref{eq:srt} for each $t$. 

\noindent ii) The cost of~\eqref{eq:srt} is \emph{not} strongly convex. In this case, we will derive dynamic regret bounds.

Before proceeding, to capture the variability of the clustering solutions, define $\sigma_t = \|C_{t+1}^*-C_t^*\|_F$~\cite{Simonetto}.
As expected, it will be shown that high variability leads to poor tracking performance. The following assumptions are then introduced.

\begin{assumption}
The matrix $X_t^TX_t$ is positive definite for each time $t$ (e.g., $N>T$ and $X_t$ is full rank).
Let $m_t>0$ be the smallest eigenvalue of $X_t^TX_t$.
\end{assumption}

By inspection, the $m_t$ of Assumption 1 is the strong convexity constant of $f_t$.
Based on the assumptions below, the first result is stated next, where $\|\cdot\|$ is taken to be the Frobenius norm.

\begin{theorem}
Under Assumption 1, 
\begin{equation}
\forall t \ge 1,
    \|C_t-C_t^*\|\leq \Tilde{L}_{t-1}\left(\|C_0-C_0^*\|+\sum_{\tau=0}^{t-1} \frac{\sigma_{\tau}}{\Tilde{L}_{\tau}}\right)
\end{equation}
where $L_t=\max\{|1-\gamma_tm_t|,|1-\gamma_tM_t|\}$, $\Tilde{L}_t=\prod_{\tau=0}^t L_{\tau}$.
\end{theorem}
\begin{proof}
Define $G_t$ and $P_t$ as in \eqref{eq:prox_online}, 
so $C_{t+1}=(P_t\circ G_t) C_t$. Observe, for any $C,C'\in\R^{T\times T}$,
\begin{align*}
    \|G&_t(C)-G_t(C')\|^2\\
    &= \|C-\gamma_t \nabla f_t(C)-C'+\gamma_t \nabla f_t(C')\|^2\\
    &\leq (1-2\gamma_t m_t + \gamma_t^2 M_t^2)\|C-C'\|^2\\
    &\leq L_t^2\|C-C'\|^2.
\end{align*}
Also, by optimality, $C_t^*$ is a fixed point, so $C_t^* = P_tG_tC_t^*$.
Therefore, one has that $\|C_{t+1}-C_t^*\| = \|P_tG_tC_t-P_tG_tC_t^*\| \leq \|G_tC_t-G_tC_t^*\| \leq L_t\|C_t-C_t^*\|$, 
where the second inequality comes from the nonexpansiveness of the prox operator~\cite{Combettes}. Finally, we get
\begin{align*}
    \|C_{t+1}-C_{t+1}^*\| &\leq \|C_{t+1}-C_t^*\|+\|C_{t+1}^*-C_t^*\|\\
    &\leq L_t\|C_t-C_t^*\|+\sigma_t
\end{align*}
and we apply this inequality recursively.
\end{proof}

\begin{corollary}
Define $\hat{L}_t = \underset{\tau=0,...,t}{\max}L_{\tau}$
and $\hat{\sigma}_t = \underset{\tau=0,...,t}{\max}\sigma_{\tau}$. If Assumption 1 holds, then, for each $t$: 
\begin{align}
    \|C_t-C_t^*\|\leq \left(\hat{L}_{t-1}\right)^t\|C_0-C_0^*\|+\frac{\hat{\sigma}_t}{1-\hat{L}_{t-1}}.
\end{align}
If $m_\tau \ge m$ and $M_\tau\le M$ and $\gamma_\tau$ is chosen $\gamma_\tau = 2/(m_\tau+M_\tau)$ for all $\tau = 0, \ldots, t$, then $\hat{L}_t\le(M-m)/(M+m) < 1$.
\end{corollary}
\begin{proof}
Follows from Thm. 1 and the geometric series.
\end{proof}

These results closely follow the analysis of \cite{Simonetto,DallAnese,BoydPrimer}, applying it to SSC.

\begin{remark} Under Assumption 1, it can be shown that $F_t(C_t)-F_t(C_t^*)\leq\frac{M_t}{2}\|C_t-C_t^*\|^2$ \cite[Thm. 10.29]{Beck}.
\end{remark}

In the following, we consider the case where the cost of~\eqref{eq:srt} is \emph{not} strongly convex. It is clear that contractive arguments cannot be utilized in this case since~\eqref{eq:prox_online} is no longer a strongly monotone operator. 
Again, we use $\|\cdot\|$ for the Frobenius norm, and $\|g\|_\infty=\sup_{x}\,|g(x)|$.

\vspace{-0mm}
\begin{theorem}
Let $\hat{C}_t^*=\frac{1}{t}\sum_{\tau=0}^{t-1}C_{\tau}^*$, $\rho_t(\tau)=\|\hat{C}_t^*-C_{\tau}^*\|$, $\delta_t=\|f_{t+1}-f_t\|_\infty$, and $M=\max_t M_t$. Set $\gamma_t=\frac{1}{M}$. Then
\vspace{-2mm}
\begin{align*}
    &\frac{1}{t}\sum_{\tau=0}^{t-1}\left(F_{\tau}(C_{\tau})-F_{\tau}(C_{\tau}^*)\right)\leq \frac{M}{2t}\|C_0-\hat{C}_t^*\|^2+\frac{1}{t}\Biggl(F_0(C_0)\\
    &\hspace{1cm}-F_{t-1}(C_t)+\sum_{{\tau}=0}^{t-2}\delta_{\tau}+\sum_{{\tau}=0}^{t-1}\rho_t({\tau})(T+\frac{M}{2}\rho_t({\tau}))\Biggr)
\end{align*}
\end{theorem}
\begin{proof}
From the descent lemma \cite[Thm. 10.16]{Beck}, we have 
\vspace{-2mm}
\begin{align*}
    \frac{1}{t}\sum_{{\tau}=0}^{t-1}\left(F_{\tau}(C_{{\tau}+1})-F_{\tau}(\hat{C}_t^*)\right)\leq \frac{M}{2t}\|C_0-\hat{C}_t^*\|^2
\end{align*}
Using the Lipschitz continuity of $\nabla f_t$, $\partial\|\cdot\|_1\in[-1,1]^{T^2}$, and the Cauchy-Schwartz inequality, we get
\vspace{-2mm}
\begin{align*}
    \frac{1}{t}\sum_{{\tau}=0}^{t-1}\left(F_{\tau}(\hat{C}_t^*)-F_{\tau}(C_{\tau}^*)\right)\leq \frac{1}{t}\sum_{{\tau}=0}^{t-1}\rho_t({\tau})(T+\frac{M}{2}\rho_t({\tau}))
\end{align*}
\vspace{-2mm}
And, rearranging, we find\vspace{-2mm}
\begin{align*}
    &\frac{1}{t}\sum_{\tau=0}^{t-1}\left(F_{\tau}(C_{\tau})-F_{\tau}(C_{{\tau}+1})\right) = \frac{1}{t}(F_0(C_0)-F_{t-1}(C_t)\\
    &\hspace{3cm}+\sum_{{\tau}=0}^{t-2}\left(f_{{\tau}+1}(C_{{\tau}+1})-f_{\tau}(C_{{\tau}+1})\right)\\ 
    &\hspace{1cm}\leq \frac{1}{t}\left(F_0(C_0)-F_{t-1}(C_t)+\sum_{{\tau}=0}^{t-2}\delta_{\tau}\right)
\end{align*}
\vspace{-2mm}
Adding these together gives us the result.
\end{proof}

\begin{corollary}
Define $\hat{\rho}_t=\underset{\tau=0,...,t-1}{\max}\rho_t(\tau)$ and\\ $\hat{\delta}_t=\underset{\tau=0,...,t-2}{\max}\delta_{\tau}$. Then
\begin{align*}
    &\frac{1}{t}\sum_{{\tau}=0}^{t-1}\left(F_{\tau}(C_{\tau})-F_{\tau}(C_{\tau}^*)\right)\leq\frac{1}{t}\biggl(\frac{M}{2}\|C_0-\hat{C}_t^*\|^2\\
    &+F_0(C_0)-F_{t-1}(C_t)\biggr)+\frac{M\hat{\rho}_t^2}{2}+T\hat{\rho}_t+\hat{\delta}_t.
\end{align*}
\end{corollary}

\noindent The term $\hat{C}_t^*$ serves as a center point of all the $C^*_\tau$ and is the most meaningful ``best overall'' point. The corollary says that if $\hat{\delta}_t$ and $\hat{\rho}_t$ are well-behaved (i.e., the function changes slowly) then, on average, $F_t(C_t)$ tracks within a constant term of $F_t(C_t^*)$.

\begin{remark}
If $(C_t^*)$ is bounded, then $\hat{\rho}_t$ converges and $\hat{C}_t^*$ has a convergent subsequence. In that case, we can replace them with their limits in the bound in Corollary 2. The bound is only meaningful, though, when the $\delta_t$'s are finite. One way to make them finite is to impose boundedness with respect to the infinity norm as another constraint in~\eqref{eq:srt}. This particular constraint can be incorporated into our closed-form proximal projection \cite[Prop.\ 24.47]{Combettes}.
\end{remark}

\begin{remark} If $f_t$ is not strongly convex (which it will not be if $T>N$), then we can add a Tikhonov term to make it strongly convex. For example, we could add $\frac{\lambda_r}{2}\|C\|^2$. While this will provide the stronger result of Theorem 1, it incurs an error in the minimizer. That is, $\|C_{r,t}-C_t^*\|\leq\|C_{r,t}-C_{r,t}^*\|+\|C_{r,t}^*-C_t^*\|$ where $C_{r,t}$ and $C_{r,t}^*$ are the regularized tracking sequence and regularized minimizer sequence respectively.
\end{remark}

\begin{remark} \label{rmk:severalIterations}
If we take more than one iteration per time step, then we can modify $(f_t)$ accordingly in order to use the results in this paper. For example, with 2 iterations per time step, define $\Tilde{f}_t=f_{\floor*{t/2}}$. Alternatively, we can modify the Theorems. For example, if we take $n_t$ iterations at time $t$, then we just have to redefine $\Tilde{L}_t=\prod_{\tau=0}^t L_{\tau}^{n_{\tau}}$ in Theorem 1.
\end{remark}

Note that we did not consider accelerating our algorithm. In the non-strongly convex case, accelerated methods rely on global structure not just local descent. Because of this, it is not obvious that adapting an accelerated method to the online setting would lead to better tracking. However, this is a further research direction that we are currently exploring.

\section{Spectral Clustering}
\label{sec:spectral}

The main factor in determining how many iterations to take in step [S1] is the ratio between the costs of steps [S1] and [S2]. There are papers that explore online spectral clustering \cite{onlinespec}, but the results require relatively small changes in the graph. There are no such guarantees here. For example, when the oldest data point is thrown out and a new one added, all connections to the previous data point are thrown out as well, new connections have to be made for all of the points that were previously connected to the old data point, and connections have to be made for the new data point. While the end result of the spectral clustering may barely change, the change in the graph is catastrophic. Thus, 
we do a full batch spectral clustering operation,
and leave a more general framework for online spectral clustering as a future research direction.

Note that the proximal operator in (2) simplifies to $\text{proj}_{\text{diag}(\cdot)=0}\circ \text{prox}_{\gamma_t\|\cdot\|_1}$. The closed form expression for the latter proximal operator is soft-thresholding: $\text{prox}_{\gamma_t\|\cdot\|_1}C = \text{sign}(C)(|C|-\gamma)_+$ component-wise where $(a)_+=\text{max}(a,0)$. Thus, the cost of each iteration of step [S1] is dominated by the gradient descent sub-step, and so the total cost of each step is $\mathcal{O}(NT^2)$ operations. If, instead of applying the proximal gradient operator just once to $C_t$, we apply it $n_t$ times, then the step at time $t$ will cost $\mathcal{O}(n_tNT^2)$ operations.

The cost of step [S2] depends on what method we use to compute the specified eigenvectors. 
An upper bound on the cost is $\mathcal{O}(T^3)$ which can be achieved by classical dense algorithms and gives \emph{all} $T$ eigenvalues $\lambda$. To find $S \ll T$ eigenvalues approximately, the power method or Lanczos iterative methods can be used \cite{NMMC}.
First consider that $L_{rw}v=\lambda v$ iff $v-D^{-1}Wv=\lambda v$ iff $D^{-1}Wv=(1-\lambda)v$. Also, by the Gershgorin disc theorem, the eigenvalues of $L_{rw}$ are in [0,2] and the eigenvalues of $D^{-1}W$ are in [-1,1]. Thus, for $D^{-1}W$, we want the eigenvectors corresponding to the eigenvalues near 1. The leading eigenpairs of $D^{-1}W$ should correspond to the trailing eigenpairs of $L_{rw}$ as long as the trailing eigenvalues of $L_{rw}$ are closer to 0 than the leading eigenvalues are to 1. If this is not the case, then we have to use the power method to compute \textit{more} than S eigenpairs, so that we can take the eigenvectors corresponding to the S largest \textit{positive} eigenvalues. The power method costs $\mathcal{O}(nnz\cdot niter\cdot S)$ in the ideal case where we do not have to compute extra eigenpairs. Here, $nnz$ denotes the number of nonzero elements and $niter$ denotes the number of iterations to reach convergence.

There are efficient methods for computing the trailing eigenpairs of $L_{rw}$ directly. In particular, \cite{3methods} found that the Jacobi-Davidson method was superior to the Lanczos method, in terms of computation time, for spectral clustering. Both  methods cost $\mathcal{O}(nnz\cdot niter)$, but $niter \approx S$ can vary.

The number of nonzeros, in the ideal case, can be estimated by the subspace dimensions, $d_i$. For a given subspace, the minimum number of points needed to represent another point, as long as it isn't cohyperplanar with a strict subset of the points, is exactly the dimension of the subspace. Thus, we can say that $nnz(C^*)\geq \mathcal{O}(\sum_i d_i T_i^t)$ where $T_i^t$ is the number of data points in $\Sc_i$ at time $t$. The same can be said of $W$, $L$, and $L_{rw}$.
In the case where each subspace has the same dimension, $d$, and the same number of points in it, $T/S$, then $nnz(C^*) \ge \mathcal{O}(dT)$ (and recall $d\le N$).
This means in the best case, [S2] costs $\mathcal{O}(dTS)$, while a single step of [S1] costs $\mathcal{O}(NT^2)$ so this suggests choosing $T \approx dS/N$ to balance the costs of the two steps; if $T$ is smaller than this, multiple steps to solve [S1] can be taken.

\section{Numerical Results}
\label{sec:numerical}

We performed tests on both synthetic data and the Yale Face Database \cite{Yale}. The synthetic data was composed of $S=10$ 5-dimensional subspaces in $R^{50}$ ($d=5$), each with 50 points. Noise was added to make the simulation more realistic. The sliding window had a capacity of $T=400$ data points and we took 100 time steps. 
The Yale Face Database has $S=38$ subspaces, and we 
let the sliding window have a capacity of $T=500$ points and took 200 time steps. The points in the Yale Face Database are in $\R^{2016}$. For both datasets, we took 50 iterations of the optimization algorithm per time step (cf.\ Remark~\ref{rmk:severalIterations}).

For the synthetic data, $T>N$, so the cost function is not strongly convex. On the other hand, for the Yale Face Database, $T<N$, so the cost function \textit{is} strongly convex. The numerical results show that for both datasets, though, the objective value trajectory converges to a region above the minimum trajectory. This can be seen in Figure 1.

\begin{figure}[htb]

\begin{minipage}[b]{.48\linewidth}
  \centering
  \centerline{\includegraphics[width=4.0cm]{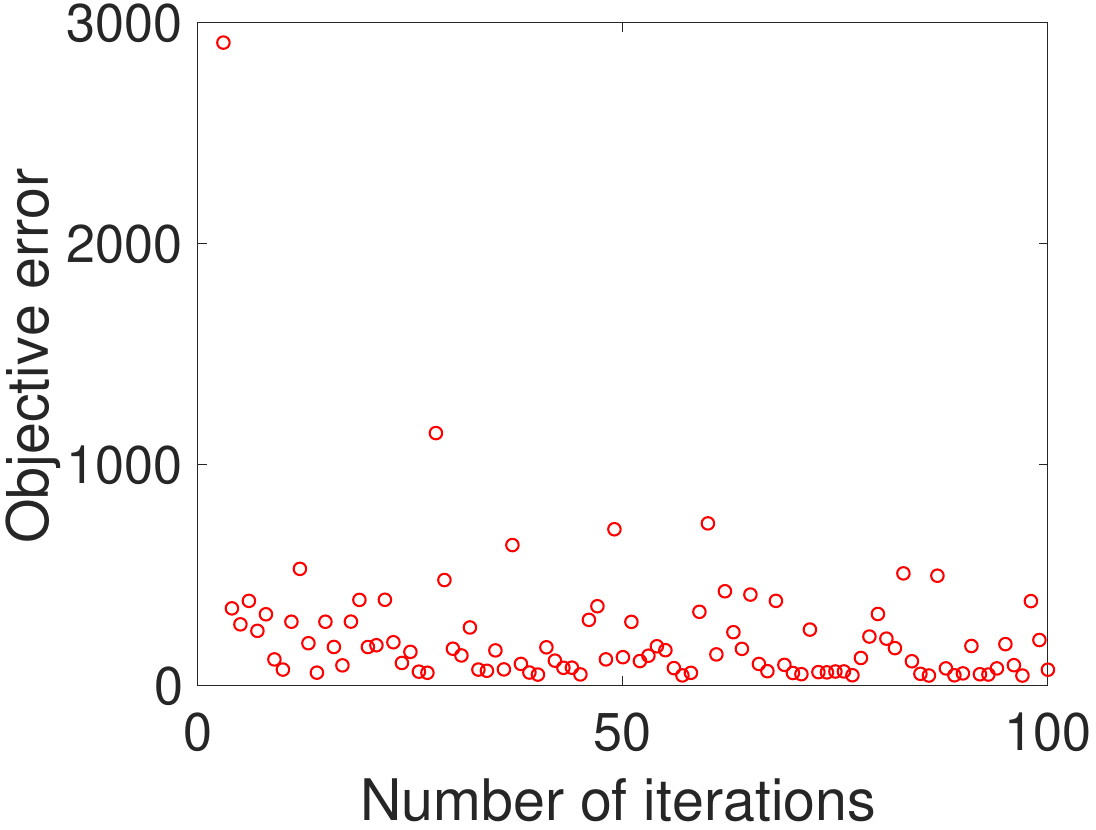}}
  \centerline{(a) Synthetic dataset}\medskip
\end{minipage}
\hfill
\begin{minipage}[b]{0.48\linewidth}
  \centering
  \centerline{\includegraphics[width=4.0cm]{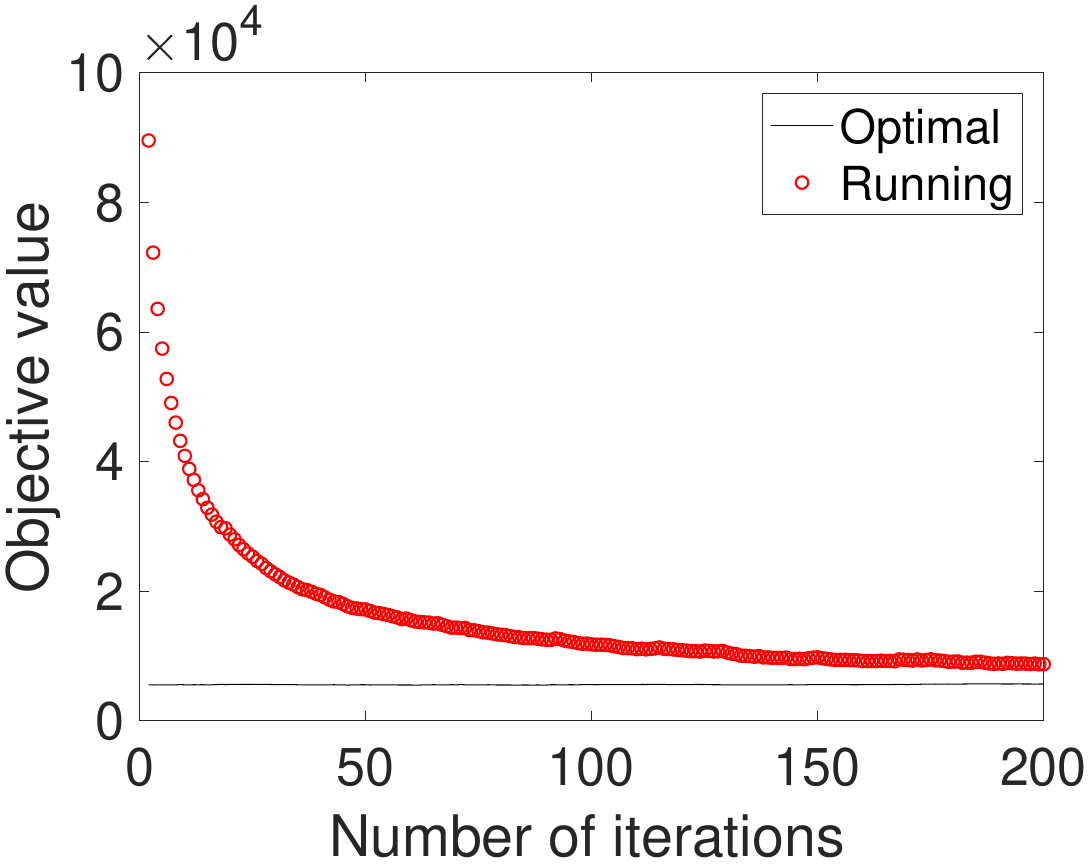}}
  \centerline{(b) Yale Face dataset}\medskip
\end{minipage}
\caption{Objective value of tracking sequence and actual time-varying minimum.}
\label{fig:res}
\end{figure}

The objective error seems to be driving whether or not the clustering error converges. Figure 2 shows the clustering error.

\begin{figure}[htb]

\begin{minipage}[b]{.48\linewidth}
  \centering
  \centerline{\includegraphics[width=4.0cm]{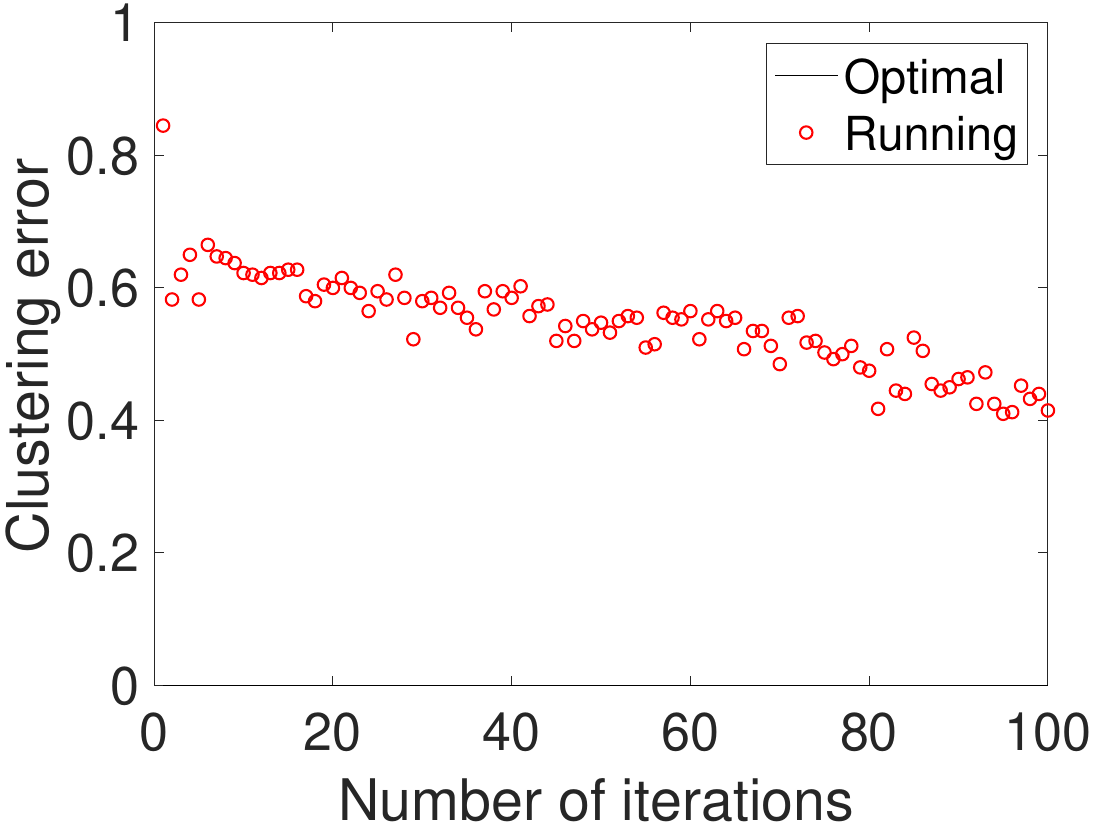}}
  \centerline{(a) Synthetic dataset}\medskip
\end{minipage}
\hfill
\begin{minipage}[b]{0.48\linewidth}
  \centering
  \centerline{\includegraphics[width=4.0cm]{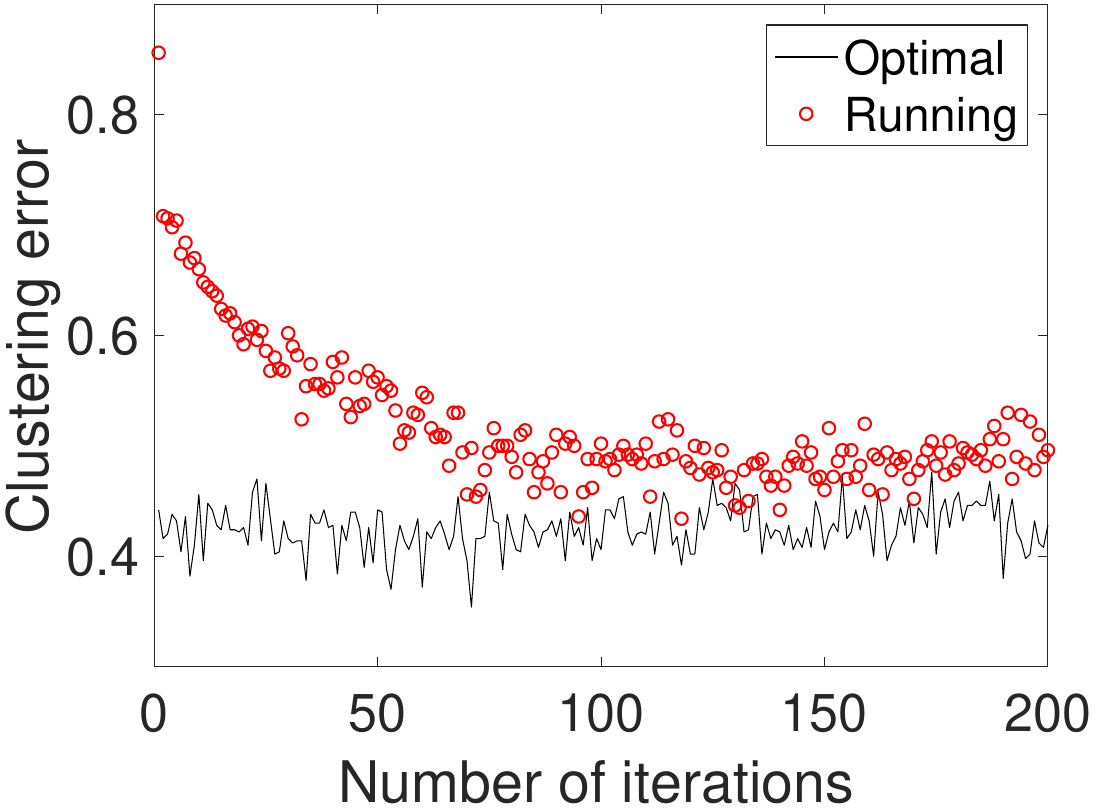}}
  \centerline{(b) Yale Face dataset}\medskip
\end{minipage}
\caption{Clustering error of both the tracking sequence and minimizer sequence.}
\label{fig:res}
\end{figure}

For both datasets, the clustering error of the tracking algorithm decreases. However, decreasing the number of iterations per time step too much causes the clustering error to no longer decrease. This suggests that there is some maximum value of the objective error for the clustering error to decrease.

Finally, for the synthetic dataset, step [S2] took 50 times as long as step [S1]. For the Yale Face dataset, step [S2] took 3 times as long as step [S1]. While we took a sufficient number of iterations in order for the clustering error of our algorithm to be small, for an actual system, its dynamics would dictate the number of iterations. The spectral clustering time would be subtracted from the length of a time step, and this value would be divided by the time for step [S1].

\bibliographystyle{IEEEbib}
\bibliography{main}
\nocite{Sleijpen}
\nocite{NystSpec1}
\nocite{NystSpec2}
\nocite{sketched}
\nocite{Nesterov}
\end{document}